\documentclass[11pt]{article}
\usepackage[latin1]{inputenc}
\usepackage{subfigure}
\usepackage{amsmath}
\usepackage{amssymb}
\usepackage{amsthm}
\usepackage{fullpage}
\usepackage{graphicx}
\usepackage{hyperref}

\newtheorem{thm}{Theorem}[section]
\newtheorem{prop}[thm]{Proposition}
\newtheorem{lem}[thm]{Lemma}
\theoremstyle{remark}
\newtheorem{rmk}{Remark}

\newcommand{\tam}[1]{$\textsc{Tam}({#1})$}
\newcommand{\tamref}{\cite{PRV2014extension}}

\author{Wenjie Fang \thanks{Email: wenjie.fang@liafa.univ-paris-diderot.fr. W.F. is partially supported by \textit{Agence nationale de la Recherche} under grant number ANR 12-JS02-001-01 ``Cartaplus'', and by the City of Paris under the grant ``\'Emergence Paris 2013, Combinatoire \`a Paris''. W.F. acknowledges hospitality from LaBRI, UMR CNRS 5800, Universit\'e de Bordeaux, France.} \\ IRIF \\ Universit\'e Paris Diderot \and Louis-Fran\c cois Pr\'eville-Ratelle \thanks{Email: preville-ratelle@inst-mat.utalca.cl. L.-F. P.-R. is supported by the government of Chile under Proyecto Fondecyt 3140298.} \\ Instituto de Mathem\'atica y F\'isica \\ Universidad de Talca}

\title{The enumeration of generalized Tamari intervals \footnote{An extended abstract of this paper was accepted by the conference Formal Power Series and Algebraic Combinatorics 2016 (FPSAC 2016), with the title ``From generalized Tamari intervals to non-separable planar maps (extended abstract)''.}}
\begin{document}
\maketitle
\begin{abstract} 
Let $v$ be a grid path made of north and east steps. The lattice \tam{v}, based on all grid paths weakly above $v$ and sharing the same endpoints as $v$, was introduced by Pr\'eville-Ratelle and Viennot (2014) and corresponds to the usual Tamari lattice in the case $v=(NE)^n$. Our main contribution is that the enumeration of intervals in \tam{v}, over all $v$ of length $n$, is given by $\frac{2 (3n+3)!}{(n+2)! (2n+3)!}$. This formula was first obtained by Tutte(1963) for the enumeration of non-separable planar maps. Moreover, we give an explicit bijection from these intervals in \tam{v} to non-separable planar maps.
\end{abstract}

\section{Background and main results} 

The well-known usual Tamari lattice can be defined on Dyck paths or some other combinatorial structures counted by Catalan numbers such as binary trees, and it has many connections with several fields, in particular in algebraic and enumerative combinatorics. In \cite{ch06}, Chapoton showed that the intervals in the Tamari lattice are enumerated by the formula 
\[ \frac{2}{n(n+1)} \binom{4n+1}{n-1}. \]
This formula also counts rooted planar 3-connected triangulations. Many results and conjectures about the diagonal coinvariant spaces of the symmetric group (we refer to the books \cite{Bergeron-book-algcomb,haglund-book} for further explanation), also called the Garsia-Haiman spaces, led Bergeron to introduce the $m$-Tamari lattice for any integer $m \geq 1$. The case $m=1$ is the usual Tamari lattice. It was conjectured in \cite{bergeron-preville} and proved in \cite{bousquet-fusy-preville} and \cite{BMCPR2013representation} that the number of intervals and labeled intervals in the $m$-Tamari lattice of size $n$ are given respectively by the formulas

\[ \frac{m+1}{n (mn+1)} \binom{(m+1)^2 n + m}{n-1} \, \, \, {\rm and} \, \, \, (m+1)^n (mn+1)^{n-2}.\]

These labeled intervals (resp. unlabeled intervals) are conjectured to be enumerated by the same formulas as the dimensions (resp. dimensions of the alternating component) of the trivariate Garsia-Haiman spaces. These connections motivated the introduction of the lattice \tam{v} in \tamref{} for an arbitrary grid path $v$ as a further generalization. In particular, the Tamari lattice of size $n$ is given by \tam{(NE)^n}, and more generally the $m$-Tamari lattice by \tam{(NE^m)^n}. A precise definition of \tam{v} will be given in Section~\ref{sec:def-interval}.

The Tamari lattice and its generalizations, while being deeply rooted in algebra, have mysterious enumerative aspects and bijective links yet to be unearthed. For instance, intervals in the Tamari lattice are equi-enumerated with planar triangulations, and a bijection was given by Bernardi and Bonichon in \cite{BB2009intervals}. Similarly, the numbers of intervals and labeled intervals in the $m$-Tamari lattice in \cite{bousquet-fusy-preville} and \cite{BMCPR2013representation} are also given by simple planar-map-like formulas, where a combinatorial explanation is still missing. In this context, similar to the bijection in \cite{BB2009intervals}, we also discover a bijection between intervals and maps, contributing to the combinatorial understanding of the Tamari lattice. We should mention that there are also bijective links between Tamari intervals, interval posets and tree flows \cite{chapoton-chatel-pons}.

In this article, we give an explicit bijection between intervals in \tam{v} and non-separable planar maps, from which we obtain the enumeration formula of these intervals. To describe it, we need two intermediate structures: one called \emph{synchronized interval}, which is a special kind of intervals in the usual Tamari lattice; the other called \emph{decorated tree}, basically a kind of rooted trees with labels on their leaves that satisfy certain conditions. The bijection from generalized Tamari intervals to synchronized intervals is implicitly given in \tamref{}. We then show that an exploration process gives a bijection between non-separable planar maps and decorated trees, and we present another bijection between decorated trees and synchronized intervals.

Tutte showed in \cite{tutte1963census} that non-separable planar maps with $n+2$ edges are counted by $\frac{2 (3n+3)!}{(n+2)! (2n+3)!}$. Therefore, we obtain as a consequence of our bijection the following enumeration formula of intervals in \tam{v}.
\begin{thm}\label{thm:enum}
The total number of intervals in \tam{v} over all possible $v$ of length $n$ is given by
\begin{equation} \label{eq:cnt} \sum_{v \in (N,E)^n} \mathrm{Int}(\textsc{Tam}(v)) = \frac{2 (3n+3)!}{(n+2)! (2n+3)!}. \end{equation}
\end{thm}

%Beside planar maps, Ch\^atel and Pons related intervals in the Tamari lattice to a new kind of combinatorial objects called \emph{interval posets} in \cite{interval-poset}. They then exploited this bijection with Chapoton in \cite{chapoton-chatel-pons} to establish some equidistribution results and further bijective relations of intervals in the Tamari lattice. 

%We would like to mention here that an extended abstract of this paper was accepted by the conference Formal Power Series and Algebraic Combinatorics 2016 (FPSAC 2016), with the title ``From generalized Tamari intervals to non-separable planar maps (extended abstract)''.

\section{From canopy intervals to synchronized intervals} \label{sec:def-interval}

A \emph{grid path} is a (finite) walk on the square grid, starting at (0,0), consisting of north and east unit steps denoted by $N$ and $E$ respectively. The size of a grid path is the number of steps it contains. For $v$ an arbitrary grid path, let \tam{v} be the set of grid paths that are weakly above $v$ and share the same endpoints as $v$. The covering relation defined as follows gives \tam{v} a lattice structure. For $v_1$ a grid path in \tam{v} and $p$ a grid point on $v_1$, we define the horizontal distance horiz$_v(p)$ to be the maximum number of east steps that we can take starting from $p$ without crossing $v$. The left part of Figure~\ref{fig:tam-def} gives an example of a path in \tam{v}, with the horizontal distance of each lattice point on $v_1$. Suppose that $p$ is preceded by a step $E$ and followed by a step $N$ in $v_1$. Let $p'$ be the first lattice point in $v_1$ after $p$ such that horiz$_v(p')$ $=$ horiz$_v(p)$, and $v_1 [p,p']$ the sub-path of $v_1$ from $p$ to $p'$. By switching the step $E$ just before $p$ and the sub-path $v_1 [p,p']$ in $v_1$, we obtain another path $v_1'$ in \tam{v}. The covering relation $\prec_v$ in \tam{v} is given by $v_1 \prec_v v_1'$, and the lattice structure of \tam{v} is given by the transitive closure $\leq_v$ of $\prec_v$. The right part of Figure~\ref{fig:tam-def} gives an example of $\prec_v$. The path $v$ is called the \emph{canopy} due to what it represents on binary trees, and it is also the minimal element in \tam{v}.

\begin{figure}[!htbp]
\centering
\includegraphics[page=9]{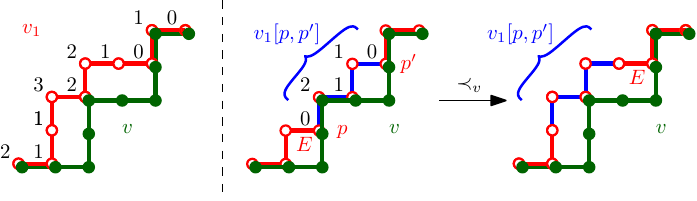}
\caption{Horizontal distance and the covering relation in \tam{v}} \label{fig:tam-def}
\end{figure}

A \emph{canopy interval} of size $n \geq 0$ is given by a triple of paths $(v_2,v_1,v)$, all of size $n$, such that $v_2 \geq v_1$ in \tam{v}. They are in bijection with a family of intervals in the usual Tamari lattice called \emph{synchronized intervals}, which are better suited for establishing the bijections in the next sections. We refer readers to \tamref{} for more details about the lattice \tam{v} and canopy intervals.

We now need the concept of Dyck paths. A \emph{Dyck path} of size $n$ is a finite walk on $\mathbb{Z}^2$, starting at (0,0), consisting of $n$ \emph{up steps} $u=(1,1)$ and $n$ \emph{down steps} $d=(1,-1)$, and never crossing the $x$-axis. We choose to use diagonal steps $(u,d)$ for Dyck paths instead of north and east steps $(N,E)$ of grid paths to underline that these two types of paths are elements in \textbf{different} lattices. It is known (\textit{cf.} \cite{XV_DeVi84, levine}) that pairs of non-crossing paths with the same endpoints of size $n-1 \geq 0$ are counted by the Catalan number $C_n = \frac{1}{n+1}\binom{2n}{n}$, which is also the number of Dyck paths of size $n$. 

For a Dyck path $P = (p_i)_{1 \leq i \leq 2n}$ where $p_i$ are steps, let $i_1, \ldots, i_n$ be the indices such that $p_{i_k} = u$. We define $\mathit{Type}(P)$ as the following word $w$ of length $n-1$: for $k \leq n-1$, if $p_{i_k} = p_{i_k+1} = u$, then $w_k = E$, otherwise $w_k = N$. Let $I(v)$ be the set of Dyck paths of type $v$. It is also an interval in the usual Tamari lattice. Types of Dyck paths are related to \tam{v} by the following theorem.

\begin{thm}[Theorem~3 in \tamref{}]\label{thm:tam-part}
The usual Tamari lattice is partitioned into intervals $I(v) \simeq \textsc{Tam}(v)$ of Dyck paths of all $2^{n-1}$ possible types $v$.
\end{thm}

\begin{figure}[!htbp]
\centering
\includegraphics[page=14, scale=0.58]{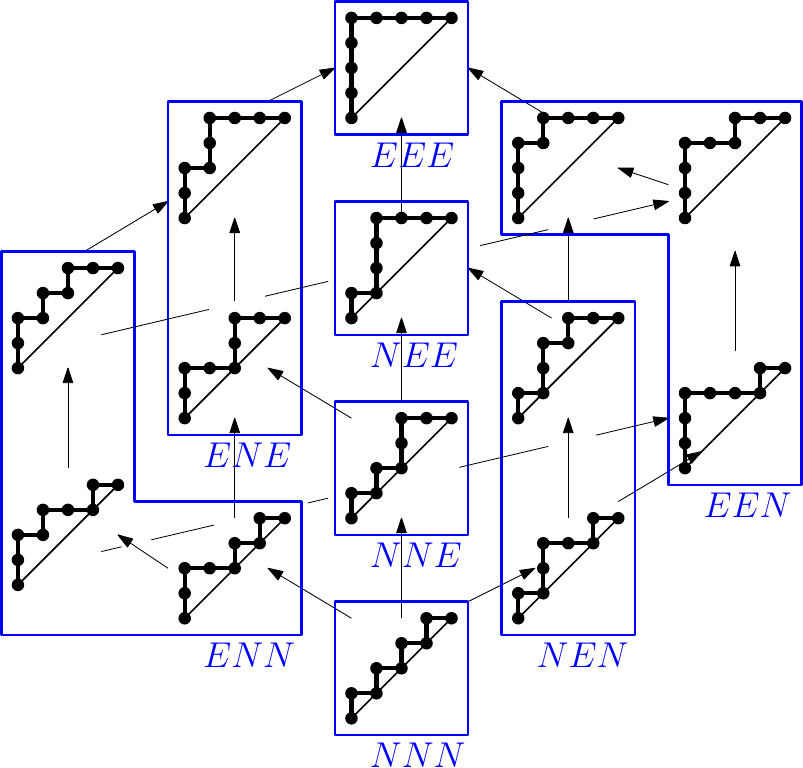} \quad \includegraphics[page=15, scale=0.58]{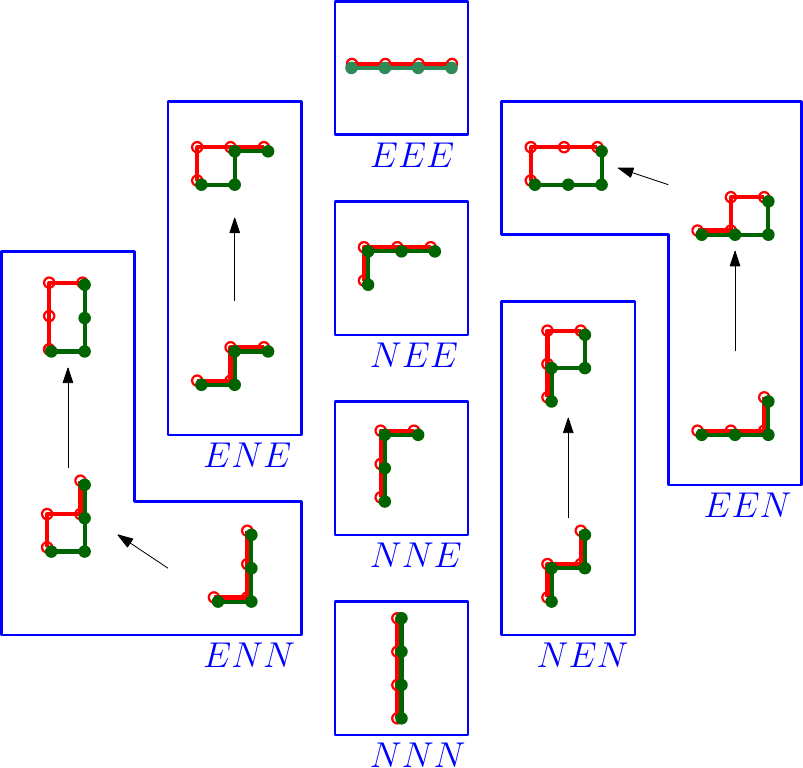}
\caption{Partition of the usual Tamari lattice of size $4$ into Dyck paths of the same type, and the corresponding generalized Tamari intervals} \label{fig:tam-part}
\end{figure}

An example of this partition by type can be seen in Figure~\ref{fig:tam-part}. A \emph{synchronized interval} is an interval inside an $I(v)$ for some grid path $v$, an interval in the usual Tamari lattice made of two Dyck paths of the same type $v$. We denote by $\mathcal{I}_n$ the set of synchronized intervals made of Dyck paths of length $n$. For any given path $v$, the explicit bijection presented in \tamref{} between pairs of non-crossing grid paths of size $n-1$ and complete binary trees with $n$ interior vertices, the latter equivalent to Dyck paths of size $n$. This bijection specializes to give a lattice isomorphism between \tam{v} and the (sub-)interval of the Tamari lattice $I(v)$.%, which induces a bijection between their respective intervals. We recall that the size of a canopy interval is the number of steps, while the size of a synchronized interval is the number of up steps. Before giving the explicit bijection in \tamref{} on paths, we need some additional notation. 

There is a natural matching between up steps and down steps in a Dyck path defined as follows: let $u_i$ be an up step of a Dyck path $P$, we draw a horizontal ray from the middle of $u_i$ to the right until it meets a down step $d_j$, and we say that $u_i$ is \emph{matched} with $d_j$. We denote by $\ell_P(u_i)$ the distance from $u_i$ to $d_j$ in $P$ considered as a word, which is defined as the number of letters between $u_i$ and $d_j$ plus 1. For example, in $P=uududd$, we have $\ell_P(u_1)=5$, since its matching letter $d$ is the one at the end. We define the \emph{distance function} $D_P$ by $D_P(i) = \ell_P(u_i)$, where $u_i$ is the $i^{\rm th}$ up step in $P$. Let $u_1$ be an up step in $P$ with its match $d_1$, and similarly $u_2$ be another up step in $P$ with its match $d_2$. Suppose that $u_1$ precedes $u_2$ in $P$. We will say $u_1$ \emph{contains} $u_2$ if $u_2$ is contained in the subpath of $P$ with both extremities $u_1$ and $d_1$.

We now describe the bijection between Dyck paths and pairs of non-crossing grid path. An example is given in Figure~\ref{fig:dyck-gen-bij}. As before, let $P$ be a Dyck path of size $n$. Let $v$ be the grid path $v=\mathit{Type}(P)$, which is of size $n-1$. We use the same terminology as in the definition of $\mathit{Type}(P)$. Let $\{i_{r_1},i_{r_2},...,i_{r_{s-1}}\}$ be the subset of  $\{ i_1,...,i_{n-1} \}$ such that $w_{r_1}=...=w_{r_{s-1}}=N$, and let $i_{r_s}=i_n$. For $1\leq k \leq s-1$, let $c_k$ be the number of up steps in $P$ that contain both the ${r_{k}}^{\rm th}$ and the ${r_{k+1}}^{\rm th}$ up steps of $P$. The grid path $v_1$ in \tam{v} is the (unique) grid path such that its $k^{\rm th}$ north step is at distance $c_k$ to the left of the $k^{\rm th}$ north step in $v$, for all $1\leq k \leq s-1$. The (translated) bijection in \tamref{} sends $P$ to the pair of grid paths $(v_1,v)$. 

\begin{figure}[!bhtp]
\begin{center}
\includegraphics[page=17,scale=0.9]{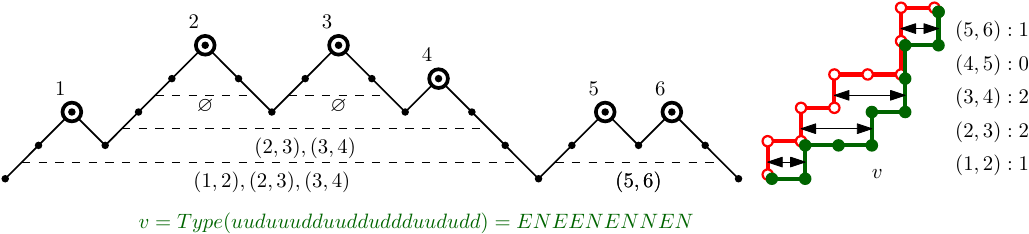}
\end{center}
\caption{An example of the bijection between Dyck paths and pairs of non-crossing grid paths} \label{fig:dyck-gen-bij}
\end{figure}

Since this bijection specializes to isomorphisms between lattices, we thus have a bijection between canopy intervals of size $n-1$ and synchronized intervals of size $n$, which implies that they are equi-enumerated. 

\section{Recursive decompositions}

We are now interested in the link between two families of objects: synchronized intervals in the usual Tamari lattice on Dyck paths of size $n$, and non-separable planar maps with $n+1$ edges. In fact, their enumerations are governed by the same functional equation. In this section, we show how to decompose recursively these two families of combinatorial objects to obtain a functional equation of their generating functions. We reiterate that our main contribution, which is the enumeration of generalized intervals via a (non-recursive) bijection, will be described explicitly in the next section.

\subsection{Recursive decomposition of synchronized intervals}

We define a \emph{properly pointed Dyck path} to be a Dyck path $P = P^\ell P^r$ such that $P^\ell$ and $P^r$ are Dyck paths, and $P^\ell$ is not empty unless $P$ is itself empty. A \emph{properly pointed synchronized interval} $[P^\ell P^r, Q]$ is a synchronized interval where the lower Dyck path is properly pointed. We recall that $\mathcal{I}_n$ is the set of synchronized intervals of size $n$, and we denote by $\mathcal{I}^\bullet_n$ the set of properly pointed synchronized intervals of size $n$. 

Before giving a recursive decomposition of synchronized intervals, we need to % introduce some additional notation and 
borrow a lemma from \cite{bousquet-fusy-preville}. 

%There is a natural matching between up steps and down steps in a Dyck path defined as follows: let $u_i$ be an up step of a Dyck path $P$, we draw a horizontal ray from the middle of $u_i$ to the right until it meets a down step $d_j$, and we say that $u_i$ is \emph{matched} with $d_j$. We denote by $\ell_P(u_i)$ the distance from $u_i$ to $d_j$ in $P$ considered as a word, which is defined as the number of letters between $u_i$ and $d_j$ plus 1. For example, in $P=uududd$, we have $\ell_P(u_1)=5$, since its matching letter $d$ is the one at the end. We define the \emph{distance function} $D_P$ by $D_P(i) = \ell_P(u_i)$, where $u_i$ is the $i^{\rm th}$ up step in $P$.

\begin{lem}[Proposition~5 in \cite{bousquet-fusy-preville}]\label{ericmirlouis}
Let $P$ and $Q$ be two Dyck paths of size $n$. Then $P \leq Q$ in the Tamari lattice if and only if $D_{P}(i) \leq D_{Q}(i)$ for all $1 \leq i \leq n$.
\end{lem}

We can now describe a way to construct a larger synchronized interval from a smaller synchronized interval and a properly-pointed synchronized interval.

\begin{prop} \label{prop:rec-construct}
Let $I_1 = [P_1^\ell P_1^r, Q_1]$ be a properly pointed synchronized interval and $I_2 = [P_2, Q_2]$ a synchronized interval. We construct the Dyck paths
\[ P = u P_1^\ell d P_1^r P_2, \qquad Q = u Q_1 d Q_2. \]
Then $I = [P,Q]$ is a synchronized interval. Moreover, this transformation from $(I_1, I_2)$ to $I$ is a bijection between $\cup_{n \geq 0} \mathcal{I}^\bullet_n \times \cup_{n \geq 0} \mathcal{I}_n$ and $\cup_{n > 0} \mathcal{I}_n$.
\end{prop}
\begin{proof} 
An illustration of the construction of $I$ is given in Figure~\ref{fig:interval-decomp}. To show that $I$ is a synchronized interval, we only need to show that $P$ and $Q$ have the same type, and $[P,Q]$ is an interval in the Tamari lattice. To show that $\mathrm{Type}(P)=\mathrm{Type}(Q)$, we notice that $P_2$ and $Q_2$ are of the same type since they form a synchronized interval. It is clear that, for two Dyck paths $P_\ell, P_r$, the type of their concatenation $P_\ell P_r$ is the concatenation of $\mathrm{Type}(P_\ell)$ and $\mathrm{Type}(P_r)$. Let $W_1 = \mathrm{Type}(u P_1^\ell d P_1^r)$ and $W_2 = \mathrm{Type}(u Q_1 d)$. We only need to show that $W_1 = W_2$. We write $W_1 = w_1 W_1'$ and $W_2 = w_2 W_2'$, where $w_1$ (resp. $w_2$) is the first letter of $W_1$ (resp. $W_2$), and $W_1'$ (resp. $W_2'$) is the rest of the word $W_1$ (resp. $W_2$). We clearly have $w_1 = w_2$, since $w_1 = N$ if and only if $P_1^\ell$ is empty, which is equivalent to $Q_1$ being empty, which occurs if and only if $w_2=N$. For the rests $W_1'$, since $P_1^\ell$ is a Dyck path that ends in $d$ by definition, we have $W_1' = \mathrm{Type}(P_1^\ell)\mathrm{Type}(P_1^r) = \mathrm{Type}(P_1^\ell,P_1^r)$. The same argument applies to $W_2'$ and $Q_1$, which leads to $W_2' = \mathrm{Type}(Q_1)$. Since $[P_1,Q_1]$ is a synchronized interval, we have $W_1'=W_2'$. We conclude that $W_1=W_2$. Therefore, $P$ and $Q$ are of the same type. It is not difficult to show that $[P,Q]$ is a Tamari interval using Lemma~\ref{ericmirlouis} and the fact that both $I_1 = [P_1^\ell P_1^r,Q_1]$ and $I_2=[P_2,Q_2]$ are also Tamari intervals.

%we only need to show that $P$ and $Q$ have the same second letter, since $P_1^\ell$ and $Q_1$ end with $d$, thus the other parts of the type inherit invariantly from $P_1, P_2$ and $Q_1, Q_2$. Since $P_1^\ell$ is a Dyck path, the second letter of $P$ is $d$ if and only if $P_1^\ell$ is empty, which is equivalent to $Q_1$ being empty, thus also equivalent to the second letter of $Q$ being $d$. Therefore, $I$ is a synchronized interval. 

%It is not difficult using Lemma~\ref{ericmirlouis} and the fact that $I=[P,Q]$ is a Tamari interval to show that $I_1=[P_1^\ell P_1^r,Q_1]$ and $I_2=[P_2,Q_2]$ are also Tamari intervals. 

To show that the transformation we described is indeed a bijection, we only need to show that we can decompose any non-empty Tamari interval $I=[P,Q]$ back into $(I_1,I_2)$. To go back from $I=[P,Q]$ to $(I_1,I_2)$, we only need to split $P$ and $Q$ into $P=u P_1^\ell d P_1^r P_2$ and $Q=u Q_1 d Q_2$ such that $P_1^\ell, P_1^r, P_2, Q_1, Q_2$ are all Dyck paths with $P_2, Q_2$ of the same length. This can be done by first cutting $Q$ at the first place that it touches again the $x$-axis, where $P$ also touches the $x$-axis. This cutting breaks $Q$ into $u Q_1 d$ and $Q_2$, and $P$ into $P_1$ and $P_2$. We then perform the same operation on $P_1$ to cut it into $u P_1^\ell d$ and $P_1^r$. We thus conclude that we indeed have a bijection between $\cup_{n \geq 0} \mathcal{I}^\bullet_n \times \cup_{n \geq 0} \mathcal{I}_n$ and $\cup_{n > 0} \mathcal{I}_n$.
\end{proof}

\begin{figure}[!htbp]
\centering
\includegraphics[page=4, scale=0.85]{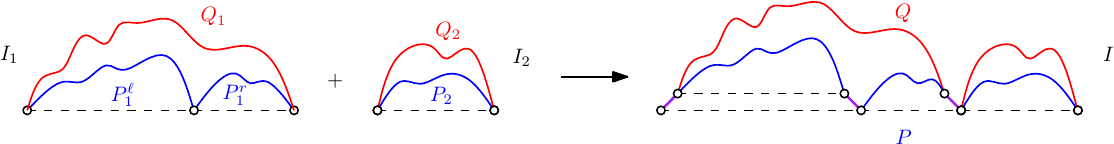}
\caption{Composition of synchronized intervals} \label{fig:interval-decomp}
\end{figure}

Since the construction in Proposition~\ref{prop:rec-construct} is a bijection, we can also see it in the reverse direction as a recursive decomposition. We now translate the recursive decomposition in Proposition~\ref{prop:rec-construct} into a functional equation for the generating function of synchronized intervals. To this end, we need to investigate another statistic on synchronized intervals, which will give us a suitable catalytic variable for our functional equation. A \emph{contact} of a Dyck path $P$ is an intersection of $P$ with the $x$-axis. Both endpoints of $P$ are also considered as contacts. Let $\operatorname{contacts}(P)$ be the number of contacts of $P$. We define $F(x,t)$ as the following generating function of synchronized intervals: 
\[ F(x,t) = \sum_{n \geq 1} \sum_{[P,Q] \in \mathcal{I}_n} t^n x^{\operatorname{contacts}(P)-1}. \]

From Proposition \ref{prop:rec-construct}, we know that a non-empty synchronized interval $I=[P,Q]$ can be decomposed into $P=u P_1^\ell d P_1^r P_2$ and $Q = u Q_1 d Q_2$, where $P_1^\ell, P_1^r, P_2, Q_1, Q_2$ are all Dyck paths, and both $[P_1^\ell P_1^r, Q_1]$ and $[P_2,Q_2]$ are synchronized intervals. The generating functions for the intervals of the form $[u P_1^\ell d P_1^r,u Q_1 d]$ is given by $xt\left( 1 + \frac{F(x,t) - F(1,t)}{x-1} \right)$, where the divided difference accounts for \emph{pointing} each non-initial contact (individually) over all elements in $\mathcal{I}_n$ to obtain properly pointed intervals of the form $I^\bullet=[P_1^\ell P_1^r,Q_1]$. Indeed, a non-empty synchronized interval $[P_1,Q_1]$ of length $2n$ with $k+1$ contacts on $P_1$ has a contribution of $x^k t^n$ to the generating function $F(x,t)$. Furthermore, there are exactly $k$ ways to turn $P_1$ into a properly pointed Dyck path $P_1^\ell P_1^r$, by cutting at one of the contacts, except the first one. When lifted from $P_1^\ell P_1^r$ to $u P_1^\ell d  P_1^r$, all contacts inside $P_1^\ell$ except the initial one are lost, which leads to $k$ Dyck paths with length $2n+2$ and a number of contacts from $2$ to $k+1$. These Dyck paths obtained from $P_1$, paired with $u Q_1 d$ to form a synchronized interval, thus have a total contribution of $t^{n+1}(x+x^2+\cdots+x^k)=xt\frac{t^n x^k - t^n}{x-1}$. We recall that the contribution of $[P_1, Q_1]$ to $F(x,t)$ is exactly $t^n x^k$. By summing over all possible $[P_1, Q_1]$, we obtain the divided difference $xt\frac{F(x,t)-F(1,t)}{x-1}$ for the contribution of the part $[P_1^\ell P_1^r, Q_1]$ in the decomposition of synchronized intervals. The extra term $xt$ is for the case where $P_1$ is empty. On the other hand, we observe that the path $P = u P_1^\ell d P_1^r P_2$  has $\operatorname{contacts}(P) -1= (\operatorname{contacts}(u P_1^\ell d P_1^r) \! - \!1) + (\operatorname{contacts}(P_2) \! - \! 1)$. Therefore, from Proposition \ref{prop:rec-construct}, we obtain the functional equation 
\begin{equation} \label{eq:interval} F(x,t) = xt \left( 1 + \frac{F(x,t) - F(1,t)}{x-1} \right) (1 + F(x,t)).
\end{equation}

\subsection{Recursive decomposition of non-separable planar maps}

We now turn to non-separable planar maps, which were first enumerated by Tutte in \cite{tutte1963census} using algebraic methods, then by Jacquard and Schaeffer in \cite{JS1998bijective} using a bijection based on their recursive decomposition. A \emph{planar map} is an embedding of a connected graph on the sphere defined up to homeomorphism, with one oriented edge called the \emph{root}. The origin vertex of the root is called the \emph{root vertex}. We call the face on the left of the root the \emph{outer face}. A planar map is called \emph{separable} if its edges can be partitioned into two sets such that only one vertex $v$ is adjacent to some edges in both sets. Such a vertex is called a \emph{cut vertex}. A \emph{non-separable planar map} is a planar map containing at least two edges that is not separable. Figure~\ref{fig:non-sep-planar-map} gives an example of such a map. Note that we exclude the two one-edge maps. 

\begin{figure}[!htbp]
\centering
\includegraphics[page=1, scale=0.75]{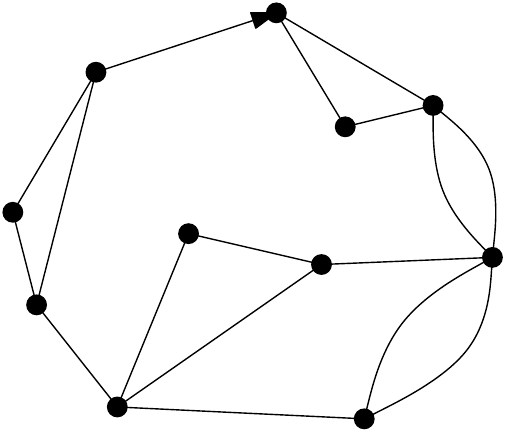} \quad \quad \quad \includegraphics[page=5, scale=0.75]{figures.pdf}
\caption{A non-separable planar map, and series/parallel decompositions of non-separable planar maps} \label{fig:non-sep-planar-map} \label{fig:map-decomp}
\end{figure}

\begin{prop}[Corollary II in \cite{tutte1963census}] \label{prop:non-sep-self-dual}
The dual of a non-separable planar map is also non-separable.
\end{prop}

There are two ways to decompose non-separable planar maps recursively. We will call them ``series'' and ``parallel'' decompositions respectively. We only need one decomposition for the functional equation, but describing both leads to a more thorough understanding. In Figure~\ref{fig:map-decomp}, we sketch how larger maps can be built with smaller maps in both series and parallel ways.

For the series decomposition of a non-separable planar map $M$, we delete its root, and the remaining map $M'$ may cease to be non-separable. In $M'$, every cut vertex splits $M'$ into two parts, each containing an endpoint of the root. The remaining map is thus a series of non-separable planar maps (and possibly single edges) linked by cut vertices (see Figure~\ref{fig:map-decomp}, middle). Let $\mathcal{M}_n$ be the set of non-separable planar maps with $n+1$ edges, and $M_s(x,t)$ the generating function of non-separable planar maps defined as
\[ M_s(x,t) = \sum_{n \geq 1} \sum_{M \in \mathcal{M}_n} t^n x^{\deg(\textrm{outer face}(M)) - 1}. \]

For a component in the series, we root it at its first edge adjacent to the outer face in clockwise order to obtain a non-separable planar map such that the root vertex is one of the linking vertices in the chain. Conversely, from a non-separable planar map $M$ with $n+1$ edges and the outer face of degree $k+1$ (therefore of contribution $t^n x^k$ in $M_s(x,t)$), there are $k$ choices for the cut vertex other than the root vertex to obtain a component, each adding a value from $1$ to $k$ to the outer face degree. These choices thus have a total contribution of $t^{n+1} \sum_{i=1}^k x^i = tx \frac{t^n x^k - t^n}{x-1}$. We recall that $M$ has a contribution $t^n x^k$ to $M_s(x,t)$. By summing over all possible $M$, we know that the generating function of these non-separable components is the divided difference $\frac{M_s(x,t) - M_s(1,t)}{x-1}$. On the other hand, the contribution of a single edge between two cut vertices is $xt$. Since the series decomposition of a non-separable map can be seen as a non-empty sequence of non-separable components and single edges, we obtain the following functional equation:
\begin{equation} \label{eq:map}
M_s(x,t) = \frac{xt + xt\frac{M_s(x,t)-M_s(1,t)}{x-1}}{1 - \left(xt + xt\frac{M_s(x,t)-M_s(1,t)}{x-1} \right)}.
\end{equation}
A reordering gives the same functional equation as \eqref{eq:interval}.

For the parallel decomposition, we consider the effect of contracting the root. Let $M'$ be the map obtained from contracting the root edge of a non-separable planar map $M$, and $u$ the vertex of the map $M'$ resulting from the contraction of the root. The only possible cut vertex in $M'$ is $u$. By deleting $u$ and attaching a new vertex of each edge adjacent to $u$, we have an ordered list of non-separable planar components (and possibly single edges) that come in parallel (see Figure~\ref{fig:map-decomp}, right). By identifying the newly-added vertices in each connected component, we obtain an ordered list of non-separable planar maps (and possibly loops). Let $M_p(x,t)$ be the generating function of non-separable planar maps as follows:
\[ M_p(x,t) = \sum_{M\in \mathcal{M}_n} t^n x^{\deg(v) - 1}. \]
Here, $v$ is the root vertex of $M$.

To obtain a non-separable component from a non-separable planar map, we only need to split the root vertex into two, that is to say to partition edges adjacent to the root vertex into two non-empty sets formed by consecutive edges. For a root vertex $v$ of degree $k$, this can be done by choosing a corner of $v$ other than the root corner, and splitting the edges by the chosen corner and the root corner. There are exactly $k-1$ choices. From a non-separable planar map with root vertex degree $k$, we can thus obtain $k-1$ different non-separable components that we can use in the parallel decomposition, each with root vertex degree from $1$ to $k-1$. We can thus write a functional equation for $M_p$, with the degree of the root vertex minus 1 as the statistics of the catalytic variable. We leave readers to check that the parallel decomposition leads to the same equation as (\ref{eq:interval}). 

Since $F, M_s, M_p$ all obey the same functional equation, we have $F = M_s = M_p$, therefore these objects are equi-enumerated under the specified statistics, which invites us to search for a bijective proof. Observe that $M_s = M_p$ already has a simple bijective explanation by duality. For instance, for a non-separable planar map $M$, we take its dual $M^\dagger$ and root it in a way such that the root vertex of $M^\dagger$ is the dual of the outer face of $M$, and the outer face of $M^\dagger$ is the dual of the root vertex of $M$. This bijection preserves the number of edges and transfers the degree of the outer face to the degree of the vertex from which the root points, which implies $M_s=M_p$.

\section{Bijections}

We now present our main contribution. To describe our bijection from synchronized intervals to non-separable planar maps, we first introduce a family of trees. We take the convention that the root of a tree is of depth $0$. The \emph{traversal order} on the leaves of a tree is simply the counter-clockwise order. A \emph{decorated tree} is a rooted plane tree with an integer label at least $-1$ on each leaf satisfying the following conditions:
\begin{enumerate}
\item For a leaf $\ell$ adjacent to a vertex of depth $p$, the label of $\ell$ is strictly smaller than $p$.
\item For each internal node of depth $p>0$, there is at least one leaf in its descendants with a label at most $p-2$.
\item For $t$ a node of depth $p$ and $T'$ a sub-tree rooted at a child of $t$, consider leaves of $T'$ in traversal order. If a leaf $\ell$ is labeled $p$ (which is the depth of $t$), each leaf in $T'$ coming before $\ell$ has a label at least $p$.
\end{enumerate}
The right side of Figure~\ref{fig:bij-map-tree} gives an example of a decorated tree. In a decorated tree, a leaf labeled with $-1$ is called a \emph{free leaf}. We denote by $\mathcal{T}_n$ the set of decorated trees with $n$ edges (internal and external).

The definition of decorated trees may not be very intuitive, but after the introduction of the exploration process, we will see that each condition captures an important aspect of non-separable planar maps, and together they characterize trees that we obtain from non-separable planar maps via the exploration process we will define.

\subsection{From maps to trees}

We start with a bijection from non-separable planar maps to decorated trees, which relies on the following exploration procedure. For a non-separable planar map $M$ with a root pointing from $v$ to $u$, we perform a depth-first exploration of vertices in clockwise order around each vertex, starting from $v$ and the root. When the exploration along an edge adjacent to the current vertex $w$ encounters an already visited vertex $x$, we replace the edge by a leaf attached to $w$ labeled with the depth of $x$ in the tree, with the convention that the depth of $v$ is $-1$. Since the map is non-separable, this exploration gives a spanning tree whose root $v$ has degree $1$, or else $v$ will be a cut vertex of the map. We then delete the edge $(v,u)$ to obtain $\mathrm{T}(M)$. Figure~\ref{fig:bij-map-tree} gives an instance of the transformation $\mathrm{T}$.

\begin{figure}[!htbp]
\centering
\includegraphics[page=2, scale=0.85]{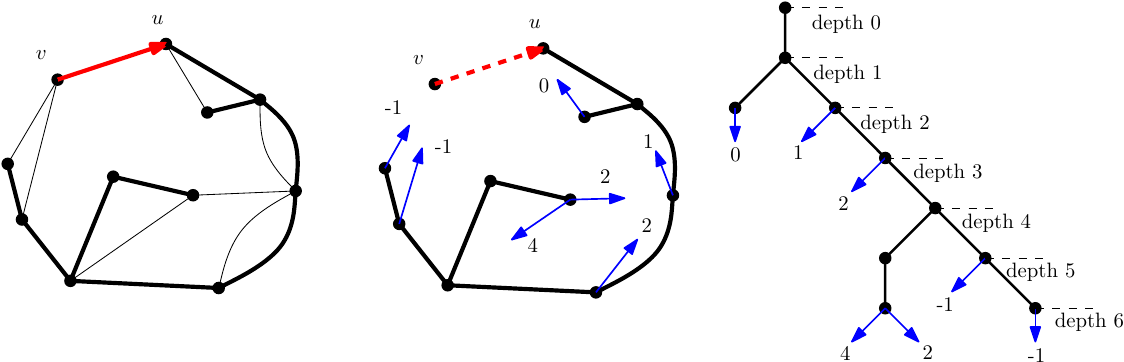}
\caption{An example of the bijection between non-separable planar maps and decorated trees} \label{fig:bij-map-tree}
\end{figure}

By abuse of notation, we identify internal nodes of $\mathrm{T}(M)$ with corresponding vertices in $M$. We notice that, for children of the same vertex in the tree, the ones being visited first in the exploration of the map come last in the traversal order. Readers familiar with graph algorithms will notice that this exploration procedure is very close to an algorithm proposed by Hopcroft and Tarjan in \cite{hopcroft} that finds 2-connected components of an undirected graph. Indeed, our exploration procedure can be seen as an adaptation of that algorithm to planar maps, where there is a natural cyclic order for edges adjacent to a given vertex.

We now present the inverse $\mathrm{S}$ of $\mathrm{T}$, with an example illustrated in Figure~\ref{fig:bij-S}. For a decorated tree $T$ rooted at $u$, we define $\mathrm{S}(T)$ as the map obtained according to the following steps.
\begin{enumerate}
\item Attach an edge $\{u,v\}$ to $u$ with a new vertex $v$, and make it the root, pointing from $v$ to $u$.
\item In clockwise order, for each leaf $\ell$ in the tree starting from the last leaf in traversal order, do the following. Let $t$ be the parent of $\ell$ and $p$ the label of $\ell$. Let $s$ be the ancestor of $\ell$ of depth $p$, and $e$ be the first edge of the path from $s$ to $\ell$ (thus an edge adjacent to $s$). We replace $\ell$ by an edge from $t$ to $s$ by attaching its other end to $s$ just after $e$ in clockwise order around $s$. See also the right-hand side of Figure~\ref{fig:ST}
\end{enumerate}

\begin{figure}[!htbp]
\centering
\includegraphics[page=12, scale=0.85]{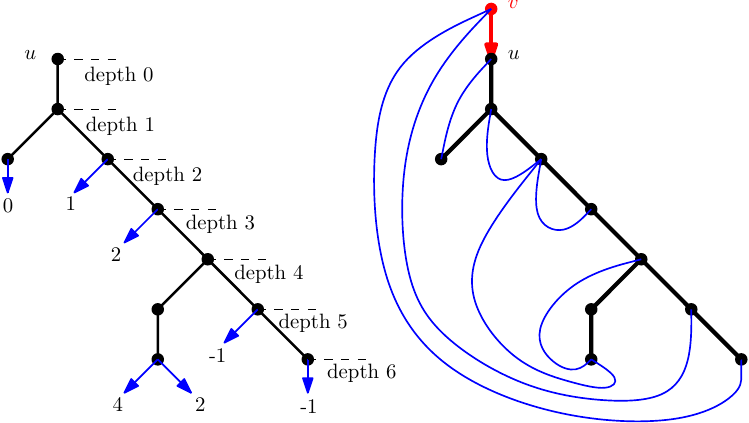}
\caption{An example of the bijection $\mathrm{S}$ from decorated trees to non-separable planar maps}
\label{fig:bij-S}
\end{figure}

From the definitions of $\mathrm{T}$, we will show that the first and third conditions of decorated trees guarantee that $\mathrm{T}$ is an exploration tree of a certain planar map, and the second guarantees the map is non-separable. We now give detailed proofs that $\mathrm{T}$ and $\mathrm{S}$ are well-defined transformations between $\mathcal{M}$ and $\mathcal{T}$, and that they are bijective and inverses of each other.

\begin{prop} \label{inclusionTMnintoTn} 
$\mathrm{T}(\mathcal{M}_n) \subset \mathcal{T}_n$.
\end{prop}
\begin{proof}
Let $M \in \mathcal{M}_n$ be a non-separable planar map with $n+1$ edges. It is clear that $\mathrm{T}(M)$ has $n$ edges. We suppose that the root of $M$ points from $v$ to $u$. We need to check the three conditions of decorated trees on $\mathrm{T}(M)$. For the first condition, let $\ell$ be a leaf adjacent to a node $t$ of depth $p$ in $\mathrm{T}(M)$, resulting from the exploration of the edge $e = \{t, t'\}$. It follows from the exploration order that $t'$ is a vertex of $M$ that was not completely explored when $e$ was visited. Thus, $t'$ is an ancestor of $t$ in $\mathrm{T}(M)$, and the label of $\ell$ is strictly smaller than $p$. For the second condition, let $t$ be a node of $\mathrm{T}(M)$ of depth $p>0$. If all leaves in the sub-tree induced by $t$ have labels at least $p-1$, then any path linking $t$ and the root vertex $v$ goes through the parent of $t$, making it a cut vertex, which is forbidden. For the third condition, let $t$ be a node of depth $p$ and $T'$ one of the sub-trees rooted at a child of $t$, and suppose that there is a leaf $\ell$ labeled $p$ in $T'$. By the exploration order, the cycle formed by the edge corresponding to $\ell$ and the path from $\ell$ to $t$ in $\mathrm{T}(M)$ encloses all leaves in $T'$ coming before $\ell$. Therefore, any leaf coming before $\ell$ in $T'$ cannot have a label strictly less than $p$, or else there will be a crossing in $M$ that makes it not planar. With all conditions satisfied, $\mathrm{T}(M)$ is a decorated tree. 
\end{proof}

\begin{prop}\label{inclusionSTnintoMn} 
$\mathrm{S}(\mathcal{T}_n) \subset \mathcal{M}_n$.
\end{prop}
\begin{proof}
Let $T \in \mathcal{T}_n$ be a decorated tree with $n$ edges. It is clear that $\mathrm{S}(T)$ is a map with $n+1$ edges. We first prove that $\mathrm{S}(T)$ is a non-separable planar map. It is not difficult to show from the first and the third conditions of the definition of decorated trees and from the definition of $\mathrm{S}$ that $\mathrm{S}(T)$ is planar. We suppose that $\mathrm{S}(T)$ is separable and $t$ is a cut vertex. We cannot have $t=v$ since $T$ is already a connected spanning tree of all vertices besides $v$, therefore $t$ is a vertex of $T$. Suppose that $t$ has depth $p$. We consider the connected component of $\mathrm{S}(T)$ containing $v$ after removing $t$. It must contain all vertices that are not descendants of $t$. Therefore, for at least one sub-tree of $t$ rooted at a child $t'$ of $t$, there is no leaf linking to ancestors of $t$, or equivalently this sub-tree only contains leaves with labels greater than or equal to $p$, which violates the definition of decorated trees since $t'$ has depth $p+1$. 
\end{proof}

\begin{prop}\label{compositionST}
  For any non-separable planar map $M$, we have $\mathrm{S}(\mathrm{T}(M)) = M$.
\end{prop}
\begin{proof}
Using leaf labels, it is clear from the definitions that $\mathrm{S}(\mathrm{T}(M))$ is equal to $M$ as a graph, and we only need to show that they have the same cyclic order of edges around each vertex. Let $t$ be an internal node of depth $p$ in $\mathrm{T}(M)$. We consider its descendant leaves of label $p$ in one of its sub-trees $T'$ induced by a descendant edge $e$ adjacent to $t$. Let $\ell_i$ be such a leaf. When reconnecting, the new edge corresponding to $\ell_i$ should come before $e$ by construction of $\mathrm{T}(M)$, and it cannot encompass other sub-trees rooted at a child of $t$, or else $t$ will be a cut vertex (see the left part of Figure~\ref{fig:ST}). If there are multiple such leaves in $T'$, their order is fixed by planarity (see the right part of Figure~\ref{fig:ST}). The reasoning also works for the extra vertex $v$ that is not in $\mathrm{T}(M)$. There is thus only one way to recover a planar map from $\mathrm{T}(M)$, and we have $\mathrm{S}(\mathrm{T}(M)) = M$. 
\end{proof}
\begin{figure}[!htbp]
\centering
\includegraphics[page=10, scale=1]{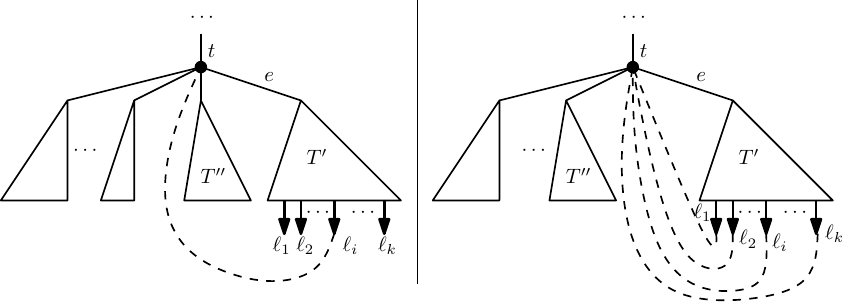}
\caption{Illustration of the proof of Prop.\ref{compositionST}} \label{fig:ST}
\end{figure}

\begin{prop}\label{compositionTS}
For any decorated tree $T$, we have $\mathrm{T}(\mathrm{S}(T)) = T$.
\end{prop}
\begin{proof}
Let $M = \mathrm{S}(T)$. We only need to show that the exploration tree $T'$ of $M$ is $T$ without labels. Closing each leaf one by one in the procedure $\mathrm{S}(T)$, it is clear that the exploration tree remains the same, therefore $T' = T$.
\end{proof}

\begin{thm}\label{bijectMnTn}
The transformation $\mathrm{T}$ is a bijection from the set of non-separable planar maps $\mathcal{M}_n$ to the set of decorated trees $\mathcal{T}_n$ for any $n>0$, and $\mathrm{S}$ is its inverse.
\end{thm}
\begin{proof}
This is a consequence of Propositions~\ref{inclusionTMnintoTn},~\ref{inclusionSTnintoMn},~\ref{compositionST}~and~\ref{compositionTS}.
\end{proof}

\subsection{From trees to intervals}

We now construct a bijection from decorated trees to synchronized intervals. For a decorated tree $T$, we want to construct a synchronized interval $[\mathrm{P}(T), \mathrm{Q}(T)]$. For the upper path, we simply define $\mathrm{Q}(T)$ as the transformation from the tree $T$ to a Dyck path by taking the depth evolution in the tree traversal. The definition of $\mathrm{P}$ is more complicated. We need to define a quantity on leaves of the tree $T$ called the \emph{charge}. The transformation $\mathrm{P}$ takes the following steps.

\begin{enumerate}
\item Every leaf has an initial charge $0$. For each internal vertex $v$ of depth $p>0$, we add $1$ to the charge of the first leaf in its descendants (in traversal order) with label at most $p-2$. We observe that the total number of charges is exactly the number of internal vertices.
\item We perform a traversal of the tree to construct a word in $u,d$. When we first visit an internal edge, we append $u$ to the word. When we first visit a leaf with charge $k$, we append $ud^{1+k}$ to the word. We thus obtain the word $\mathrm{P}(T)$.
\end{enumerate}

\begin{figure}
\centering
\includegraphics[page=3, scale=0.9]{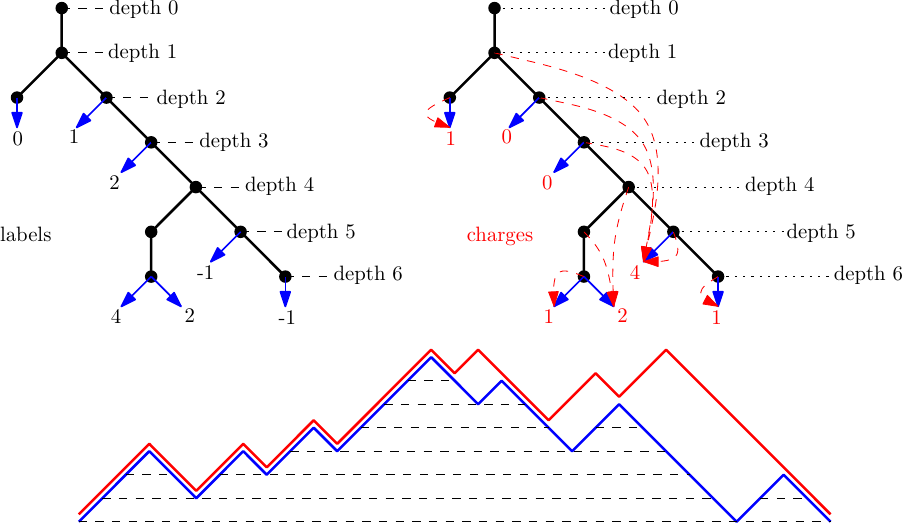}
\caption{An example of a decorated tree $T$, with the charges on its leaves and the corresponding interval $[\mathrm{P}(T),\mathrm{Q}(T)]$} \label{fig:bij-tree-interval}
\end{figure}

An example of the whole process is given in Figure~\ref{fig:bij-tree-interval}. We now prove that the transformations $\mathrm{P}$ and $\mathrm{Q}$ send a decorated tree to a synchronized interval.

\begin{prop} \label{prop:PQ-conform}
For a decorated tree $T$, the paths $\mathrm{P}(T)$ and $\mathrm{Q}(T)$ are Dyck paths, and $[\mathrm{P}(T), \mathrm{Q}(T)]$ is a synchronized interval.
\end{prop}
\begin{proof}
Since $\mathrm{Q}(T)$ is the depth evolution of the traversal of $T$, it is a well-defined Dyck path of length $2n$, where $n$ is the number of edges in $T$, which equals to the number of internal vertices of $T$. By the charging process, there are $n$ charges on leaves in total, and it is clear that $\mathrm{P}(T)$ is also of length $2n$ with $n$ up steps. We need to show that $\mathrm{P}(T)$ is positive. Consider a letter $d$ in $\mathrm{P}(T)$. The charge that gives rise to this letter $d$ comes from a non-root vertex $t$ and goes onto a descendant leaf $\ell$ of $t$. Let $e$ be the edge from $t$ to its parent. We pair up this letter $d$ to the letter $u$ in $\mathrm{P}(T)$ given by a traversal of $e$. All letters in $\mathrm{P}(T)$ can be paired up in this way, and by the traversal rule, in a pair, the letter $u$ always comes before the letter $d$ since $\ell$ is a descendant of $t$. Therefore, $\mathrm{P}(T)$ is positive, thus also a Dyck path. We can also easily see that $\mathit{Type}(\mathrm{P}(T)) = \mathit{Type}(\mathrm{Q}(T))$, since in both $\mathrm{P}(T)$ and $\mathrm{Q}(T)$, a letter $u$ is followed by a letter $d$ if and only if it corresponds to a leaf.

We now need to show that $[\mathrm{P}(T), \mathrm{Q}(T)]$ is a Tamari interval. Let $u_i^{Q}$ be the $i^{\rm th}$ up step in $\mathrm{Q}(T)$, and $e$ the edge in $T$ that gives rise to $u_i^{Q}$ in the construction of $\mathrm{Q}(T)$. By the definition of $\mathrm{P}$ and $\mathrm{Q}$, it is clear that $e$ also gives rise to the $i^{\rm th}$ up step $u_i^{P}$ in $\mathrm{P}(T)$. If we can show that $D_P(i) \leq D_Q(i)$ for all $i$, then by Lemma~\ref{ericmirlouis}, we know that $[\mathrm{P}(T), \mathrm{Q}(T)]$ is a Tamari interval.

Let $v$ be the lower endpoint of the edge $e$, $T'$ the sub-tree of $T$ rooted at $v$, $\ell$ the descendant leaf that $v$ charges and $p$ the depth of $v$. By the charging process, $\ell$ has a label at most $p-2$. Let $d_j^{P}$ be the matching down step of $u_i^{P}$ in $\mathrm{P}(T)$. We prove that $d_j^{P}$ is generated during the traversal of $\ell$, from which it follows that $D_P(i) \leq D_Q(i)$ by the definition of the distance function. Let $k \geq 1$ be the number of charges of $\ell$. Consider the segment $P(e,\ell) = u_i^{P}Wd^{k+1}$ of $P$ from $e$ to $\ell$ in the traversal for the construction of $\mathrm{P}(T)$. We first show that $|P(e,\ell)|_u \leq |P(e,\ell)|_d$, which implies that $d_j^P$ is in $P(e,\ell)$. To this end, we only need to show that we can always pair an up step in $P(e,\ell)$ with a down step also in $P(e,\ell)$. An up step is generated either by an internal edge or a leaf. For an up step $u_*$ generated by an internal edge $e'$ visited in $P(e,\ell)$, let $v'$ be its lower endpoint and $\ell'$ the leaf charged by $v'$. We know that $\ell' = \ell$ or $\ell'$ precedes $\ell$, therefore the down step $d_*$ produced by the charge added to $\ell'$ by the internal vertex $v'$ is already in $P(e,\ell)$. We pair $u_*$ with $d_*$. For an up step arisen from visiting a leaf, we can pair it with the first letter $d$ given by the leaf. Therefore, $|P(e,\ell)|_u \leq |P(e,\ell)|_d$. We now show that $d_j^{P}$ is not in $W$. Indeed, since $T$ is a decorated tree and $\ell$ has a label at most $p-2$, by the third condition of decorated trees, a descendant leaf $\ell'$ of $v$ that precedes $\ell$ has a label at least $p-1$. All charges of $\ell'$ come from internal nodes in $T'$ other than $v$ and they are all ancestors of $\ell'$, which means that they are visited before $\ell'$ in the traversal. Therefore, there are more $u$'s than $d$'s in any prefix of $W$, so $d_j^{P}$ cannot be in $W$, thus it must be among the down steps $d^{k+1}$ produced by $\ell$, which completes the proof.
\end{proof}

We now describe the inverse transformation $\mathrm{R}$, which sends a synchronized interval $[P,Q]$ to a decorated tree $T = \mathrm{R}([P,Q])$ by the following steps. A partial example is illustrated in Figure~\ref{fig:bij-tree-interval-leaf}. We should note that the following definition of $\mathrm{R}$ does not use the notion of charges in the definition of $\mathrm{P}$. However, we will use the notion of charges to prove that $\mathrm{R}$ is indeed the inverse transformation of $[\mathrm{P},\mathrm{Q}]$. More precisely, we will show how to read off, from a synchronized interval, the vertices that charge a given leaf on the corresponding decorated tree without any knowledge on the labels.

\begin{enumerate}
\item We construct the tree structure of $T$ from $Q$.
\item We perform the following procedure on each leaf, as illustrated in Figure~\ref{fig:bij-tree-interval-leaf}. Let $\ell$ be a leaf. Suppose that $\ell$ gives rise to the $i^{\rm th}$ up step in $Q$. We look at the lowest point $u$ of the consecutive down steps that come after the $i^{\rm th}$ up step in $P$, and we draw a ray from $u$ to the left until intersecting the midpoint of two consecutive up steps in $P$. Suppose that the lower up step is the $j^{\rm th}$ up step in $P$. We take $e$ the edge in $T$ that gives rise to the $j^{\rm th}$ up step in $Q$. Let $p$ be the depth of the shallower end point of $e$. We label the leaf $\ell$ with $p$. In the case that no such intersection exists, $\ell$ is labeled $-1$.
\end{enumerate}

\begin{figure}[!htbp]
\centering
\includegraphics[page=13, scale=1]{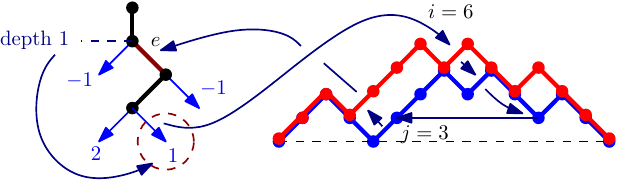}
\caption{An example of how to recover leaf labels using the lower path $P$ (here the leaf with label $1$)} \label{fig:bij-tree-interval-leaf}
\end{figure}

We start by a property of Tamari intervals, which is a corollary of Lemma~\ref{ericmirlouis}.

\begin{lem}\label{lem:sync-inc}
Let $[P,Q]$ be a Tamari interval. For $A \in \{P,Q\}$, we denote by $u_j^A$ the $j^{\rm th}$ up step in $A$, $d_j^A$ its matching down step, and $A^{[j]}$ the segment of $A$ between $u_j^A$ and $d_j^A$, excluding both ends. For any index $i, j$ such that the $i^{\rm th}$ up step $u_i^P$ of $P$ is in $P^{[j]}$, the $i^{\rm th}$ up step $u_i^Q$ in $Q$ is also in $Q^{[j]}$.
\end{lem}
\begin{proof}
We observe that, for any Dyck path $A$ and index $j$, the segment $A^{[j]}$ is a Dyck path of size $(D_A(j)-1)/2$. Since $A^{[j]}$ is formed by consecutive letters in $A$, the up steps in $A^{[j]}$ have consecutive indices starting from $j+1$. Therefore, $u_i^A$ is in $A^{[j]}$ if and only if $j+1 \leq i \leq j+1+(D_A(j)-1)/2$. By Lemma~\ref{ericmirlouis}, $D_P(j) \leq D_Q(j)$. Therefore, $u_i^P$ is in $P^{[j]}$ implies that $u_i^Q$ is in $Q^{[j]}$.
\end{proof}

In the following proofs, for a tree $T$, the sub-tree induced by an edge $e$ is the sub-tree obtained by cutting $e$.

\begin{prop}\label{prop:R-conform}
For $[P,Q]$ a synchronized interval, the tree $T = \mathrm{R}([P,Q])$ is a decorated tree.
\end{prop}
\begin{proof}
We need to verify that $T$ satisfies the three conditions of decorated trees. We first look at the first condition for the case $P=Q$. In this case, to show that the label of a leaf $\ell$ attached to a vertex $u$ of depth $p$ is strictly smaller than $p$, we consider the next newly visited edge $e$ after $\ell$ in the tree traversal. If no such $e$ exists, $\ell$ is labeled $-1$. Let $v$ be the vertex adjacent to $e$ with smaller depth, and $v$ must be an ancestor of $\ell$. In the case $P=Q$, the label of $\ell$ is the depth of $v$ minus $1$, which must be strictly smaller than $p$. The first condition is also satisfied for any other $P$ since going down in the lattice weakly reduces the labels in $\mathrm{R}([P,Q])$. Therefore, the first condition is satisfied for all $\mathrm{R}([P,Q])$.

For the second condition, let $v$ be an internal node of depth $p$ in $T$ that is not the root and $e$ the edge from $v$ to its parent. We need to exhibit one descendant leaf of $v$ that has a label of value at most $p-2$. Suppose that $e$ corresponds to the $j^{\rm th}$ up step $u_j^Q$ in $Q$. On the path $P$, let $u_j^P$ be the $j^{\rm th}$ up step and $d_j^P$ be the matching down step. Let $d_i^P$ be the first down step of the consecutive down steps containing $d_j^P$, and $u_i^P$ be its matching up step, which is also the $i^{\rm th}$ up step of $P$. It is clear that $u_i^P$ is between $u_j^P$ and $d_j^P$ in $P$. Let $\ell$ be the leaf that gives rise to $u_i^Q$. We now show that the label of $\ell$ is at most $p-2$. Since $[P,Q]$ is a synchronized interval, thus also a Tamari interval, by Lemma~\ref{lem:sync-inc}, the $i^{\rm th}$ up step $u_i^Q$ of $Q$ must be between the $j^{\rm th}$ up step $u_j^Q$ and its matching down step in $Q$. Furthermore, the edge $e$ gives rise to $u_j^Q$, therefore $\ell$ must be a descendant leaf of $e$. Consider the lowest down step $d_k^P$ corresponding to $\ell$ in $P$. Let $e'$ be the edge that gives label to $\ell$. By the definition of $\mathrm{R}$, the segment of $P$ between the corresponding up step of $e'$ and $d_k^P$ contains $u_j^P$, which implies that the edge $e'$ induces a sub-tree containing $e$ (or the entire tree when the label is $-1$) according to Lemma~\ref{lem:sync-inc}. Let $w$ be the endpoint with a smaller depth of $e'$. The depth of $w$ is thus at most that of $v$ minus $2$, therefore the label of $\ell$ is at most $p-2$. The tree $T$ thus satisfies the second condition.

For the third condition, let $\ell$ be a leaf in $T$ with a label $p$. We consider the edge $e$ that gives a label to $\ell$ in the construction of $T$. The edge $e$ links two vertices of depth $p$ and $p+1$. Therefore, its induced sub-tree $T'$ is one of the sub-trees of a vertex $v$ of depth $p$ as in the third condition. The leaf $\ell$ is in $T'$ by the construction of $T$ and Lemma~\ref{lem:sync-inc}. We only need to prove that there is no leaf with label strictly less than $p$ before $\ell$ in $T'$. Let $\ell_1$ be a leaf in $T'$ that comes before $\ell$. We suppose that $\ell$, $e$ and $\ell_1$ give rise to the $i^{\rm th}$, $j^{\rm th}$ and $i_1^{\rm th}$ up step in $Q$ respectively. We have $j < i_1 < i$. We now look at the corresponding up steps in $P$. By construction of the horizontal ray, the lowest point of the consecutive down steps that comes after the $i_1^{\rm th}$ up step in $P$ cannot be lower than that of the $i^{\rm th}$ step, or else the ray would be blocked. Therefore, $\ell_1$ receives a label from an edge in $T'$ or from $e$, which give it a label at least $p$. The third condition is thus satisfied, and we conclude that $T$ is a decorated tree.
\end{proof}

We now show that the transformations $[\mathrm{P}, \mathrm{Q}]$ and $\mathrm{R}$ are inverse of each other. 

\begin{prop}\label{prop:inverse-R}
Let $[P,Q]$ be a synchronized interval and $T = \mathrm{R}([P,Q])$, we have $\mathrm{P}(T)=P$ and $\mathrm{Q}(T) = Q$.
\end{prop}
\begin{proof}
For $Q$ it is clear. We only need to prove the part for $P$. From Proposition~\ref{prop:R-conform}, we know that $[\mathrm{P}(T),\mathrm{Q}(T)]$ is a synchronized interval. Therefore, given the path $Q$, the Dyck path $\mathrm{P}(T)$ is totally determined by the charge of each leaf in $T$. We only need to show that each leaf in $T$ receives the correct amount of charge, which is one less than the length of the corresponding consecutive down steps in $P$.

We will first investigate the charging process. Let $v$ be a non-root vertex of depth $p$ in $T$, $e$ the edge linking $v$ to its parent, and $u_i^P$ the up step in $P$ that comes from $e$, which is also the $i^{\rm}$ up step of $P$. Let $\ell$ be the leaf that gives rise to the matching down step $d_i^P$ of $u_i^P$ in $P$. We now show that $v$ charges $\ell$ by showing that $\ell$ has a label of value at most $p-2$ and showing that $\ell$ has the first such label.

To show that the label of $\ell$ is at most $p-2$, we consider the labeling process on $\ell$. We draw a ray to the left from the lowest point of the down steps of $\ell$, which hits a double up step. Suppose that the lower one of the double up step is the $j^{\rm th}$ up step $u_j^P$ of $P$. Let $e'$ be the edge giving rise to $u_j^P$, and $v'$ the endpoint of $e'$ with a smaller depth. It is clear that $i \neq j$ and $u_i^P$ is in the segment of $P$ between $u_j^P$ and its matching down step. Therefore, by Lemma~\ref{lem:sync-inc}, $e$ must also be in the sub-tree induced by $e'$. Therefore, $v'$ is of depth at most $p-2$, thus $\ell$ has a label at most $p-2$. 

To show that $\ell$ is the first descendant leaf of $v$ with a label at most $p-2$, we consider a descendant leaf $\ell'$ of $v$ that comes before $\ell$ in the traversal order. Let $d_k^P$ be the last down steps in $P$ that comes from $\ell'$. Since $\ell'$ comes before $\ell$, $d_k^P$ is strictly between $u_i^P$ and $d_i^P$, and the horizontal ray from the lower point of $d_k^P$ lays strictly above that from $u_i^P$ to $d_i^P$. By the labeling process, the double up steps that corresponds to $d_k^P$ (thus to $\ell'$) is in the segment from $u_i^P$ to $d_i^P$ (both ends included), therefore the label of $\ell'$ is at least $p-1$. We conclude that $\ell$ is the first leaf in the sub-tree of $v$ that has a label at most $p-2$, and thus $v$ charges $\ell$.

To count the number of charges of $\ell$, we notice that each down step in $P$ that comes from $\ell$ corresponds to a charge, except for the highest one. To see this, we only need to consider their matching up steps. It is clear that highest down step of $\ell$ in $P$ is matched with the only up step in $P$ that comes from $\ell$. For a down step of $\ell$ that is not the highest, it is impossible that its matching up step is immediately followed by a down step, therefore the matching up step is the lower part of a double up step, corresponding to an internal vertex in $T$, and we can see from the argument above that this internal vertex charges $\ell$. We thus conclude that $\ell$ receives the correct number of charges in the construction of $\mathrm{P}(T)$, which implies $\mathrm{P}(T)=P$.
\end{proof}

We now show that the transformation $[\mathrm{P},\mathrm{Q}]$ is an injection.

\begin{prop}\label{prop:PQ-injective}
Let $T_1, T_2$ be two decorated trees. If $\mathrm{P}(T_1) = \mathrm{P}(T_2)$ and $\mathrm{Q}(T_1) = \mathrm{Q}(T_2)$, then $T_1 = T_2$.
\end{prop}
\begin{proof}
Suppose that $T_1 \neq T_2$. Since $\mathrm{Q}(T_1) = \mathrm{Q}(T_2)$, the only difference between $T_1$ and $T_2$ must be on labels. Let $\ell$ be the first leaf in the traversal order that $T_1$ and $T_2$ differ in labels. We suppose that $\ell$ has a label $k_1$ in $T_1$ and label $k_2$ in $T_2$, with $k_1 > k_2 \geq -1$. We have $k_1 \geq 0$. It is clear that all nodes charging $\ell$ in $T_1$ also charge $\ell$ in $T_2$. Since $\mathrm{P}(T_1) = \mathrm{P}(T_2)$, we know that $\ell$ receives the same number of charges in $T_1$ and $T_2$, thus $\ell$ is also charged by the same set of vertices in $T_1$ and $T_2$. Let $u$ be the ancestor of $\ell$ of depth $k_1 + 1$ in $T_1$ and $T_2$. The vertex $u$ has a parent since $k_1 \geq 0$. The existence of $u$ is guaranteed by the first condition of decorated trees. In $T_1$, by construction, $u$ does not charge $\ell$, therefore $u$ should not charge $\ell$ in $T_2$ either. Therefore, in $T_2$ there must be a descendant leaf $\ell'$ of $u$ that has a label at most $k_2$ and comes before $\ell$, to prevent $u$ from charging $\ell$. By the minimality of $\ell$, we know that $\ell'$ also has a label $k_2 < k_1$ in $T_1$, violating the third condition of decorated trees on the parent of $u$, which is impossible. Therefore, $T_1 = T_2$.
\end{proof}

We now prove that $\mathrm{R}$ is a bijection between decorated trees and synchronized intervals.

\begin{thm}
The transformation $\mathrm{R}$ is a bijection from $\mathcal{I}_n$ to $\mathcal{T}_n$ for all $n \geq 1$, with $[\mathrm{P}, \mathrm{Q}]$ its inverse.
\end{thm}
\begin{proof}
It is clear that $\mathrm{R}$ preserves the size $n$. By Proposition~\ref{prop:PQ-conform} and Proposition~\ref{prop:R-conform}, we have $[\mathrm{P},\mathrm{Q}](\mathcal{T}_n) \subset \mathcal{I}_n$ and $\mathrm{R}(\mathcal{I}_n) \subset \mathcal{T}_n$. By Proposition~\ref{prop:inverse-R} and Proposition \ref{prop:PQ-injective}, both transformations are injective, therefore they are bijections between $\mathcal{I}_n$ and $\mathcal{T}_n$, and they are the inverse of each other.
\end{proof}

By composing the two bijections $\mathrm{T}$ and $[P,Q]$, we obtain a natural bijection from non-separable planar maps with $n+1$ edges to synchronized intervals of size $n$ via decorated trees. By the equivalence between synchronized intervals and canopy intervals described in Section \ref{sec:def-interval} (see also \cite{PRV2014extension}), we then have a bijection between non-separable planar maps and canopy intervals. Therefore, these two kinds of objects are enumerated by the same formula. Figure~\ref{fig:full-bij} shows an example of how this chain of bijections gets from generalized Tamari intervals to non-separable planar maps.

\begin{figure} \label{fig:full-bij}
\centering
\includegraphics[page=16,scale=0.9]{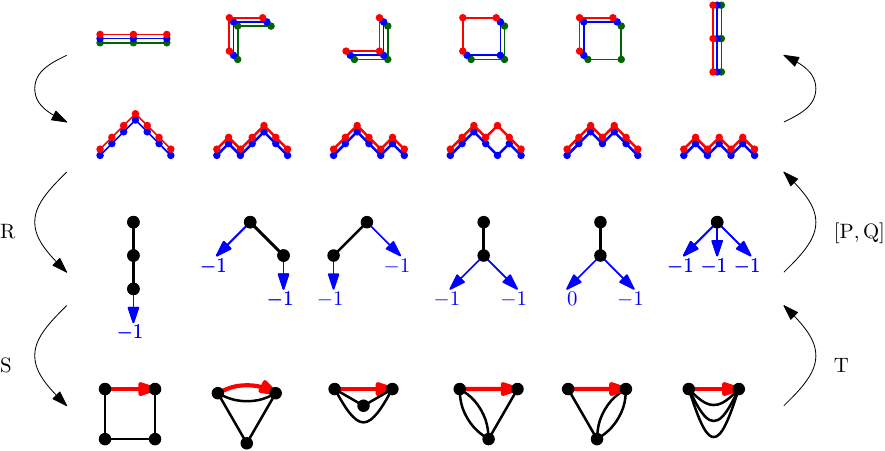}
\caption{The series of bijections from generalized Tamari intervals of length $2$ to non-separable planar maps with $4$ edges}
\end{figure}

Our bijection has further structural implications on the generalized Tamari lattice and non-separable maps, which will be the subject of an up-coming follow-up article. We also have the following remark.

\begin{rmk}
There is a version of parallel decomposition  of non-separable planar maps that just removes one non-separable component in the rightmost part of Figure~\ref{fig:map-decomp}. This decomposition is isomorphic to the one underlying Proposition~\ref{prop:rec-construct} for synchronized intervals. There thus exists a ``canonical'', recursively defined bijection between $\mathcal{M}_n$ and $\mathcal{I}_n$. We can prove that it coincides with our bijection $[P,Q] \circ \mathrm{T}$.
\end{rmk}

\section{Discussion}

Our bijection allows us not only to enumerate canopy intervals, but also to investigate various statistics in these two kinds of objects, both coming with interesting hidden symmetries. It is also interesting to ask for other significant statistics that are also transferred by our bijection and other interesting natural involutions.

We know that \tam{(NE)^n} is the usual Tamari lattice, and in \cite{BB2009intervals}, Bernardi and Bonichon gave a bijection between Tamari intervals and planar rooted triangulations. It is thus natural to look for a similar bijection as a specialization of our bijection, and eventually a generalization to \tam{(NE^m)^n}, which is isomorphic to the $m$-Tamari lattice. In general, searching for a bijection between intervals in the $m$-Tamari lattice and a natural class of planar maps is of particular interest, as it would provide an insight on the reason that the enumeration formulas of intervals in the $m$-Tamari lattice and their recursive decompositions are similar to those of various types of planar maps. For now, there doesn't exist such a natural class of planar maps for $m>1$. In the case $m=1$, we can prove bijectively the enumeration formula of unlabeled intervals by first relating them bijectively to rooted triangulations as in \cite{BB2009intervals}, which can then be enumerated bijectively as in \cite{triang-bij}. However, no such bijective proof exists for the formula of their labeled variant (\textit{cf.} \cite{BMCPR2013representation}). There is hope that discovering a natural class of planar maps for arbitrary $m$ would permit the use of bijective methods in maps (\textit{cf.} \cite{Schaeffer:survey}) to construct a (hopefully uniform) bijective interpretation of the formulas for unlabeled and labeled intervals (and more generally for the character formulas, \textit{cf.} \cite{BMCPR2013representation}). In another direction, using trees with blossoms to encode planar maps bijectively is a common practice, but we rarely see a DFS tree in this interplay. Triangulations may be a step towards the extension of this new approach.

We recall that we have used the intermediate structure of ``trees with charges on leaves'' in the bijection $[\mathrm{P},\mathrm{Q}]$ from decorated trees to synchronized intervals. As pointed out by an anonymous referee, this structure is in fact in bijection with the so-called \emph{closed flows} on trees, first proposed in \cite{tree-flow} by Chapoton for research in dendriform algebra. More precisely, in the language of \cite{chapoton-chatel-pons}, given a rooted tree $T$, decorated trees obtained by adding labels of leaves of $T$ are in bijection with closed flows on $T$ where each internal node (except the root) is given an input $-1$. This bijection is presented in our work: given a decorated tree, we first compute the charges of each leaf, then construct the corresponding closed flow by regarding charges as input on leaves and adding an input of $-1$ on each internal node other than the root vertex. This bijection is reminiscent to the one between Tamari intervals and forest flows discovered in \cite{chapoton-chatel-pons}. 
%The implication of this bijection will be explored in a follow-up article.

As a final remark, in Chapter 2 of \cite{kitaev2011pattern}, there is a sequence of bijections that starts with 2-stacks sortable permutations and ends with non-separable planar maps, which goes through 8 intermediate families of objects. It would be interesting to see how our work is related to these equi-enumerated objects.

\section*{Acknowledgement}

The two authors would like to thank Mireille Bousquet-M\'elou, Guillaume Chapuy, \'Eric Fusy, Luc Lapointe and Xavier Viennot for many fruitful discussions. We would also like to thank the anonymous referees for their comments and for pointing out the bijective relation between decorated trees and closed flows on trees.

\nocite{*}
\bibliographystyle{alpha}
\bibliography{tamari-carte}
\label{sec:biblio}

\end{document}